\documentclass{amsart}

\usepackage{amsmath}
\usepackage{amssymb}
\usepackage{mathrsfs}
\usepackage{hyperref}
\usepackage{enumitem}
\usepackage{stmaryrd}
\usepackage[all]{xy}
\usepackage{verbatim}
\usepackage{fullpage}

\newtheorem{Thm}{Theorem}

\newtheorem{Cor}[Thm]{Corollary}
\newtheorem{Def}[Thm]{Definition}

\newtheorem{Lem}[Thm]{Lemma}

\newtheorem{Rmk}[Thm]{Remark}

\numberwithin{equation}{Thm}

\newcommand{\FF}{\mathbb{F}}

\newcommand{\ZZ}{\mathbb{Z}}

\newcommand{\GL}{\operatorname{GL}}

\newcommand{\PSL}{\operatorname{PSL}}

\author{Christopher Davis}
\address{University of California, Irvine, Dept of
Mathematics, Irvine, CA 92697}
\email{davis@math.uci.edu}
\date{\today}

\author{Tommy Occhipinti}
\address{University of California, Irvine, Dept of
Mathematics, Irvine, CA 92697}
\email{tocchipi@math.uci.edu}

\title{Which finite simple groups are unit groups?}

\begin{document}

\begin{abstract}
We prove that
if $G$ is a finite simple group which is the unit group of a ring, then $G$ is isomorphic to either (a) a cyclic group of order 2; (b) a cyclic group of prime order $2^k -1$ for some $k$; or (c) a projective special linear group $\PSL_n(\FF_2)$ for some $n \geq 3$.  Moreover, these groups do (trivially) all occur as unit groups.   We deduce this classification from a more general result, which holds for groups $G$ with no non-trivial normal 2-subgroup.
\end{abstract}

\maketitle

The finite groups $G$ of odd order which occur as unit groups of rings\footnote{Our rings are assumed unital but not necessarily commutative, and ring homomorphisms send $1$ to $1$.}  were determined in \cite{Dit71}.  We will prove similar results for a more general class of groups; the description of this class of groups uses the following.    

\begin{Def}
For a finite group $G$, the \emph{$p$-core} of $G$ is the largest normal $p$-subgroup of $G$.  We denote this subgroup by $O_p(G)$.   It is the intersection of all Sylow $p$-subgroups of $G$.  
\end{Def}

We now state the main result.  The authors\footnote{The authors also thank Colin Adams, John F.\ Dillon, Dennis Eichhorn, Noam Elkies, Kiran Kedlaya, Charles Toll, and Ryan Vinroot for many useful discussions.} are most grateful to the anonymous referee for our earlier paper \cite{DO13}, who recognized that one of the results proved in that paper could be strengthened into the following. 

\begin{Thm} \label{main theorem}
Let $G$ denote a finite group such that $O_2(G) = \{1\}$ and such that $G$ is isomorphic to the unit group of a ring $R$.  Then 
\[
G \cong \GL_{n_1} (\FF_{2^{k_1}}) \times \cdots \times \GL_{n_r}(\FF_{2^{k_r}}).
\]
\end{Thm}

Before proving Theorem~\ref{main theorem}, we record the following corollary.

\begin{Cor}
The finite simple groups which occur as unit groups of rings are precisely the groups
\begin{enumerate}[label=(\alph*)]
\item $\ZZ/2\ZZ$,
\item $\ZZ/p\ZZ$ for a Mersenne prime $p = 2^{k}-1$,
\item $\PSL_n(\FF_2)$ for $n \geq 3$.
\end{enumerate}
\end{Cor}

\begin{proof}
If $G$ is a finite simple group, then either $O_2(G) = \{1\}$ or $O_2(G) = G$.  If $O_2(G) = G$, then $G$ is a $2$-group, and because we are assuming $G$ is simple, we must have $G \cong \ZZ/2\ZZ$, which for instance is isomorphic to the unit group of $\ZZ$.  

Hence assume $G$ is a finite simple group which is isomorphic to the unit group of a ring and further assume $O_2(G) = \{1\}$.   
By Theorem~\ref{main theorem}, we know
\[
G \cong \GL_{n_1}(\FF_{2^{k_1}}) \times \cdots \times \GL_{n_r}(\FF_{2^{k_r}}).
\]
Clearly these groups all occur as unit groups, so we are reduced to determining which of them are simple;  this forces
\[
G \cong \GL_n(\FF_{2^k}).
\]
If $n > 1$ and $k > 1$, then the subgroup of invertible scalar matrices forms a nontrivial normal subgroup.  Hence two possibilities remain.  If $n = 1$, then $\GL_1(\FF_{2^k})$ is cyclic of order $2^k-1$.  If $k =1$, then $\GL_n(\FF_2) = \PSL_n(\FF_2)$.  This completes the proof.  
\end{proof}

\begin{Rmk}
The simple groups $A_8$ and $\PSL_2(\FF_7)$ also occur as unit groups.  This follows immediately from the exceptional isomorphisms
\[
A_8 \cong \PSL_4(\FF_2) \text{ and } \PSL_2(\FF_7) \cong \PSL_3(\FF_2).
\]
See for instance \cite[Section~3.12]{Wil09}.
\end{Rmk}

Having recorded the above consequences of the main result, we now gather the preliminary results used in its proof.  We begin with the following easy exercise.  

\begin{Lem} \label{trivial 2-core implies characteristic 2}
Let $G$ denote a finite group with $O_2(G) = \{1\}$, and let $R$ denote a ring with $R^{\times} \cong G$.  Then $R$ has characteristic 2.
\end{Lem}

\begin{proof}
The elements $1$ and $-1$ are units in $R$ and are in the center of $R$, hence are in the center of $R^{\times}$.  By the assumption $O_2(G) = \{1\}$, the center of $G$ cannot contain any elements of order~2.  Hence $1 = -1$.
\end{proof}

\begin{Lem} \label{ideal exists}
Keep notation as in Lemma~\ref{trivial 2-core implies characteristic 2}, and fix an isomorphism $R^{\times} \cong G$.  
The image of the natural map
\[\varphi: \FF_2[G] \rightarrow R\]
is a ring with unit group isomorphic to $G$.  
\end{Lem}

\begin{proof}
Write $S$ for the image of $\varphi$.  It is clear on one hand that $S^{\times} \subseteq R^{\times} \cong G$, and on the other hand it is clear that the induced map $\varphi: G \rightarrow S^{\times} \rightarrow R^{\times}$ is surjective.  
\end{proof}

\begin{Lem} \label{square zero ideal}
Let $R$ denote a finite ring of characteristic~2.  If $J \subseteq R$ is a two-sided ideal such that $J^2 = 0$, then $1 + J$ is a normal elementary abelian $2$-subgroup of $R^{\times}$.
\end{Lem}

\begin{proof}
Note that for any $j \in J$ and $r \in R^{\times}$, we have $(1 + j)^2 = 1 + j^2 = 1$ and 
\[r(1+j)r^{-1} = 1 + rjr^{-1} \in 1 + J.\]
\end{proof}

We now use these preliminary results to prove our main theorem. 

\begin{proof}[Proof of Theorem~\ref{main theorem}]
By Lemma~\ref{ideal exists}, we may assume $R$ is a finite ring (and is in particular artinian) and has characteristic~2.  Let $J$ denote a two-sided ideal of $R$ such that $J^2 = 0$.  By Lemma~\ref{square zero ideal}, the set $1 + J$ is a normal $2$-subgroup of $R^{\times}$, and so by the assumption $O_2(G) = \{1\}$, we have $J = \{0\}$.   Thus the ring $R$ has no non-zero two-sided ideals $J$ with $J^2 = 0$, and hence $R$ has no non-zero two-sided nilpotent ideals.  By \cite[Theorem~5.4.5]{Coh89}, the artinian ring $R$ is semisimple.    By Wedderburn's Theorem \cite[Theorem~5.3.4]{Coh89}, we have
\[
R \cong M_{n_1}(D_1) \times \cdots \times M_{n_r}(D_r)
\]
for some $n_1, \ldots, n_r \geq 1$ and some division algebras $D_1, \ldots, D_r$.  Our ring $R$ is finite and hence each $D_i$ is finite. By another theorem of Wedderburn \cite[Theorem~3.8.6]{Coh89}, we have that each $D_i$ is a finite field.  Finally, because the ring $R$ has characteristic~2, each field $D_i$ has characteristic~2.  This completes the proof.
\end{proof}

\bibliography{units}
\bibliographystyle{plain}

\end{document}